\documentclass[12pt]{amsart} 
\usepackage{graphicx}

\scrollmode

\newcommand{\calA}{\mathcal{A}}

\newcommand{\calG}{\mathcal{G}}
\newcommand{\calH}{\mathcal{H}}

\newcommand{\calK}{\mathcal{K}}

\newcommand{\calS}{\mathcal{S}}
\newcommand{\calW}{\mathcal{W}}
\newcommand{\reals}{{\mathbb R}}
\newcommand{\naturals}{{\mathbb N}}

\newcommand{\bx}{{\bf x}}
\newcommand{\by}{{\bf y}}

\newcommand{\bk}{{\bf k}} 

\newcommand{\e}{{\varepsilon}}
\newcommand{\setu}{{\mathfrak{u}}}

\newcommand{\setv}{{\mathfrak{v}}}
\newcommand{\rd}{{\rm d}}
\newcommand{\il}{\left<}
\newcommand{\ir}{\right>}
\newcommand{\newton}[2]{\left(\begin{array}{c} #1\\ #2\end{array}\right)}

\theoremstyle{plain}
\newtheorem{theorem}{Theorem}

\newtheorem{proposition}{Proposition}

\theoremstyle{definition}
\newtheorem{remark}{Remark}

\title{On Tractability of Approximation \\
       for a Special Space of Functions}

\author{M. Hegland and G. W. Wasilkowski}
\date{\today}
\begin{document}

\begin{abstract}
We consider approximation problems for a special space 
of $d$ variate functions. 
We show that the problems have small number of active variables, 
as it has been postulated in the past using {\em concentration of measure} 
arguments.  We also show that, 
depending on the norm for measuring the error, 
the problems are strongly polynomially or quasi-polynomially tractable 
even in the model of computation where functional evaluations have 
the cost exponential in the number of active variables. 
\end{abstract}

\maketitle

\section{Introduction}
This paper is inspired by \cite{HegPes05}, where an importance of 
a special class of multivariate functions was advocated, and by 
recent results on tractability of problems dealing with 
infinite-variate functions, see 
\cite{CDMK,Gne10,HMNR10,KSWW09,NH10,NHMR10,PlaWas10,Was11,WW10a,WW10b}, 
where the cost of an algorithm depends on the number of 
active variables that it uses. 

The selection of functions in \cite{HegPes05} was based on a particular 
choice of the metric used in the space of the variables $x_i$ of the 
functions and on the smoothness of the functions. Here we consider 
the case where the $x_i$ denote features of some objects. Adding new 
features will increase the distance in general, and this increase can 
grow substantially with the dimension. For example, if $x_i\in [0,1]$ 
for $i=1,\ldots,d$ then the average squared Euclidean distance of 
two points grows proportional to the dimension $d$:
$$
   \int_{[0,1]^d}\int_{[0,1]^{d}} \sum_{i=1}^d (x_i-y_i)^2\, 
     \rd\bx\,\rd\by = O(d).
$$
This unbounded growth shows that Euclidean distance cannot 
approximate any distance function between two objects for large 
$d$. This is why it was suggested in~\cite{HegPes05} to use a 
scaled Euclidean distance to characterize the dissimilarity of 
two objects based on features $x_1,\ldots,x_d$:
\[
  {\rm dist}(\bx,\by) = \sqrt{\frac{1}{d}\sum_{i=1}^d (x_i-y_i)^2}.
\]

The continuity of functions considered in \cite{HegPes05} 
was Lipschitz-continuity based on the scaled 
Euclidean distance. For differentiable functions, this leads to conditions 
of bounded 
\[
  \sum_{i=1}^d d\cdot\left(\frac{\partial f}{\partial x_i}\right)^2\leq
  L_1
\]
where $L_1$ is the Lipschitz constant of $f$ with respect to the scaled
Euclidean distance. A model example is the mean function
\[
  f(\bx) = \frac{1}{d}\sum_{i=1}^d x_i.
\]
This function has a Lipschitz constant $L_1=1$. Consequently the gradient
satisfies
$$  
   \|\nabla f\|_{L_2} \leq \frac{1}{d^{1/2}}.
$$
It follows that $f$ is approximated with an $O(d^{-1/2})$ error 
by the constant $0.5$, i.e., the values of $f$ are concentrated 
around $0.5$. This concentration phenomenon for general 
Lipschitz-continuous functions was established by L\'evy 
in~\cite{Levy}. 

Higher order approximations can be derived in the case when higher order
Lipschitz constants are finite, i.e., where for some $m>0$ one has
\[
  \sum_{i_1\leq \cdots \leq i_m}d^m\cdot 
   \left(\frac{\partial^m f}{\partial x_{i_1}
       \cdots \partial x_{i_m}}\right)^2 \leq L_m^2.  
\]
Using the example of the mean function one has
$$
  \left(\frac{1}{d}\sum_{i=1}^d x_i - \frac{1}{2}\right)^2 = O(1/d).
$$
From this one gets the first order (additive function) approximation
$$
  \frac{2}{d^2} \sum_{i<j} x_i x_j = \frac{1}{d}\sum_{i=1}^d (1-x_i/d) x_i
  -1/4 + O(1/d).
$$
A similar approximation is obtained for the average squared distance
$\frac{2}{d(d-1)} \sum_{i<j} (x_i-x_j)^2$. Both these functions do 
satisfy a higher order Lipschitz condition with respect to the 
scaled norm introduced earlier. 

Classes of such functions and the particular scaling by $1/d^m$, 
where $m$ is equal to the number of involved variables, are 
related to the weighted reproducing kernel 
Hilbert space $\calH_d$ of multivariate functions on 
$[0,1]^d$ with the reproducing kernel given by 
\[
  \calK(\bx,\by)=1+\sum_{\setu\not=\emptyset}d^{-|\setu|}\cdot
    \prod_{j\in\setu} \min(x_j,y_j).
\]
Here the sum is over all subsets $\setu$ of $\{1,\dots,d\}$. 
This is why we consider such spaces in the current paper.
It is well known, see, e.g., \cite{KSWW09b}, that functions from that 
space have an ANOVA-like representation of the form 
\[
  f(\bx)=f_\emptyset+\sum_{\setu\not=\emptyset} f_\setu(\bx),
\]
where each component $f_\setu$ depends on, exactly, the variables listed in 
$\setu$. Hence, $\setu$ is the list of active variables in 
$f_\setu$ and the scaling parameter $m$ is equal to $|\setu|$. 
The corresponding norm of $f$ is given by 
\[
  \|f\|_{\calH_d}^2=|f_\emptyset|^2+\sum_{\setu\not=\emptyset}
     d^{|\setu|}\cdot\left\|\frac{\partial^{|\setu|} f_\setu}{
      \prod_{j\in\setu}\partial x_j}\right\|_{L_2}^2. 
\]

As already mentioned, it was also postulated in \cite{HegPes05} 
that functions of this 
form are well approximated by sums of those components $f_\setu$ that 
depend on small numbers of variables, i.e., 
with $\setu$ of small cardinality, or just 
by a constant function. 
We show in a more quantified way, that this is true for approximation 
problems with errors measured in a norm of another Hilbert space $\calG_d$ 
that also has a tensor product form. 
That is, we show that to approximate $f$ with an error not 
exceeding $\e\cdot\|f\|_{\calH_d}$, it is enough to consider only those terms 
$f_\setu$ that depend on at most $|\setu|\le m(\e,d)$ variables, where 
$m(\e,d)$ grows with $1/\e$ very slowly and/or decreases to zero 
when $d$ tends to infinity. 
More precisely, for general 
tensor product spaces (including the $L_2$ space), we have 
\[
   m(\e,d)\le \min\left(d\,,\,
     \frac{c\cdot\ln(1/\e)}{\ln(\ln(1/\e))}\right). 
\]
for a known constant $c>0$ that does not depend on $\e$ and $d$.
For instance, for any $d\in\naturals_+$ and the error demand 
$\e=10^{-q}$, we have 
\[
   m\left(10^{-2},d\right)\le5,\quad 
   m\left(10^{-4},d\right)\le8,\quad \mbox{and}\quad 
   m\left(10^{-8},d\right)\le14.
\] 
Suppose next that spaces $\calH_d$ and $\calG_d$ satisfy the 
following assumption: there exists $C<\infty$ such that 
\begin{equation}\label{new-ass}
  \|f\|_{\calG_d}^2\le C\cdot\sum_{\setu}
    \|f_\setu\|_{\calG_d}^2\quad\mbox{for all}\quad
    f=\sum_{\setu}f_\setu\in\calH_d.
\end{equation}
Then $m(\e,d)$ has even a smaller upper bound
\[
  m(\e,d)\le \min\left(d\,,\,\frac{2\cdot\ln(1/\e)}{\ln(d/c)}\right). 
\]
Hence for a fixed error demand $\e$, $m(\e,d)=O\left(1/\ln(d)\right)$ 
as $d\to\infty$. 

Actually, we prove these results for reproducing kernels of the form
\[
    \calK(\bx,\by)=1+\sum_{\setu\not=\emptyset}d^{-|\setu|}\cdot
    \prod_{j\in\setu} K(x_j,y_j)
\]
for a general class of univariate kernels $K:D\times D\to\reals$ 
including of course $K(x,y)=\min(x,y)$ and $D=[0,1]$.

We also study the tractability of approximation problems 
for algorithms that can use arbitrary linear functional evaluations. 
However, as it has been done in the recent study of infinite-variate 
problems, we assume that the cost of each such evaluation depends on
the number $k$ of active variables and is given by $\$(k)$. 
Under the general tensor product assumption, approximation is 
quasi-polynomially tractable, and it is strongly polynomially 
tractable if \eqref{new-ass} is satisfied. These 
results hold even when the cost function $\$$ is exponential. 
We also find a sharp upper bound on the exponent of strong 
tractability. 
Approximation is weakly tractable even when $\$$ is doubly exponential.

\section{Basic Definitions}

\subsection{Space of $d$-Variate Functions}
Let $D\subseteq\reals$ be a Borel measurable set and let $H=H(K)$ be a 
reproducing kernel Hilbert space (RKH space for short) of 
functions $f:D\to\reals$  whose kernel is denoted by $K$. 

We assume that 
\[
  1\notin H,
\]
where $1$ denotes the constant function $f(x)=1$ for all $x$. 

In what follows we write $[1..d]$ to denote the set of positive integers 
not exceeding $d$, 
\[
  [1..d]:=\{n\in\naturals_+\ :\ n\le d\}
\]
and use $\setu,\setv$ to denote subsets of $[1..d]$. Consider now the weights 
\begin{equation}\label{weights}
  \gamma_{d,\setu}:=d^{-|\setu|} \quad\mbox{for}\quad \setu\subseteq[1..d].
\end{equation}
Clearly $\gamma_{d,\emptyset}=1$. 

The weighted space of $d$-variate functions $f:D^d\to\reals$ 
under the consideration is the RKH space $\calH_d$ whose kernel is 
given by 
\[
  \calK_d(\bx,\by):=\sum_{\setu\subseteq[1..d]}\gamma_{d,\setu}
  \cdot K_\setu(\bx,\by)   \quad\mbox{and}\quad 
    K_\setu(\bx,\by)=\prod_{j\in\setu}K(x_j,y_j)
\]
with the convention that $K_\emptyset\equiv1$. 

For each $\setu$, by $H_\setu$ we denote the RKH space whose kernel 
is equal to $K_\setu$. Clearly $H_\emptyset={\rm span}\{1\}$ and $H_\setu
\simeq H^{\otimes |\setu|}$ for $\setu\not=\emptyset$. 
It is well known that the spaces $H_\setu$, as subspaces of 
$\calH_d$, are mutually orthogonal and any $f\in\calH_d$ has the 
unique representation
\[
  f(\bx)=\sum_{\setu\subseteq[1..d]}f_\setu(\bx) \quad\mbox{with}\quad 
  f_\setu\in H_\setu 
\]
and
\[
  \|f\|_{\calH_d}^2=\sum_{\setu\subseteq[1..d]}\|f_\setu\|_{\calH_d}^2
  =\sum_{\setu\subseteq[1..d]}\gamma_{d,\setu}^{-1}\cdot\|f_\setu\|_{H_\setu}^2.
\]
This representation is similar to the {\em ANOVA} decomposition 
since each term $f_\setu$ depends only on the variables 
listed in $\setu$. 
The space considered in \cite{HegPes05} and mentioned in the
Introduction is related to space $\calH_d$ with 
the classical Wiener kernel discussed in the following example.

\vskip 1 pc\noindent{\bf Example.\ }
Consider
\[
   D=[0,1]\quad \mbox{and}\quad K(x,y)=\min(x,y).
\]
Then $H$ is the space of functions $f:[0,1]\to \reals$ that vanish at 
zero, are absolutely continuous, and have $f'\in L_2([0,1])$. 
The norm in $H$ is given by 
\[
  \|f\|_H^2=\int_0^1|f'(x)|^2\,\rd x. 
\]
For $\setu\not=\emptyset$, $H_\setu$ consists of functions that 
depend only on the variables $x_j$ with $j\in\setu$, are zero 
if at least one of those variables is zero, have the 
mixed first-order partial derivatives bounded in the $L_2$ norm, 
and
\[
   \|f\|_{\calH_d}^2=|f({\bf 0})|^2+
      \sum_{\setu\not=\emptyset}d^{|\setu|}\int_{D^d}
\left|\prod_{j\in\setu}\frac{\partial}
     {\partial x_j} f([\bx;\setu])\right|^2\,\rd \bx\quad\mbox{for\ }
     f\in\calH_d,
\]
where $[\bx;\setu]$ is given by 
\[
   [\bx;\setu]=[y_1,\dots,y_d] \quad\mbox{with}\quad
    y_j:=\left\{\begin{array}{ll} x_j & \mbox{if $j\in\setu$},\\
                 0 & \mbox{otherwise.}\end{array}\right.
\]

\subsection{Function Approximation Problems}
For every $d\ge1$, let $\calG_d$ be a separable Hilbert space 
of functions on $D^d$ such that $\calH_d$ is continuously embedded in it. 
We denote the corresponding embedding operator by  $\calS_d$, i.e., 
\[
  \calS_d:\calH_d\to\calG_d\quad\mbox{and}\quad \calS_d(f)=f.
\]
We assume that $\calS_d$ and $\calG_d$ have  tensor product forms, i.e., 
for every $\setu$ and every $f(\bx)=\prod_{j\in\setu}f_j(x_j)$ with 
$f_j\in H$, we have 
\begin{equation}\label{ass-big}
  \|f\|_{\calG_d}=\prod_{j\in\setu}\|f_j\|_{\calG_1}.
\end{equation}
For simplicity of presentation we also assume that 
\[
   \|1\|_{\calG_1}=1 \quad\mbox{so that}\quad \|1\|_{\calG_d}=1.
\]
The continuity of $\calS_d$ is equivalent to 
continuity of~$\calS_1$. Indeed, let 
\begin{equation}\label{norm-S}
  C_0:=\sup_{\|f\|_H\le1}\|f\|_{\calG_1}<\infty.
\end{equation}
Then for every $\setu$ we have 
\[
   \sup_{\|f\|_{H_\setu}\le1}\|f\|_{\calG_d}=C_0^{|\setu|}
\]
and 
\[
  \|\calS_d\|^2\le\sum_{\setu\subseteq[1..d]}\gamma_{d,\setu}\cdot 
     C_0^{2|\setu|}=\sum_{k=0}^d\newton{d}{k}\cdot d^{-k}\cdot C_0^{2\cdot k}
     =\left(1+\frac{C_0^2}d\right)^d,
\]
since 
\[
   \|f\|_{\calG_d}^2\le 
\bigg(\sum_{\setu\subseteq[1..d]}C_0^{|\setu|}\cdot\|f_\setu\|_{H_\setu}\bigg)^2
\le \sum_{\setu\subseteq[1..d]}\gamma_{d,\setu} \cdot C_0^{2\cdot|\setu|}\cdot
   \|f\|_{\calH_d}^2.
\]
Clearly 
\[
  1\le \|\calS_d\|\le e^{C_0^2/2}\quad\mbox{for every $d$}
\]
which means that the corresponding approximation 
problem is properly scaled. 

Note also that the condition \eqref{new-ass} holds if 
\begin{equation}\label{new-ass2}
  \il 1, f\ir_{\calG_1}=0\quad\mbox{for all\ }f\in H.
\end{equation}
Actually, under \eqref{new-ass2} we have 
\[
   \|f\|_{\calG_d}^2=\sum_{\setu\subseteq[1..d]} 
     \|f_\setu\|_{\calG_d}^2\quad\mbox{for all\ }f\in\calH_d.
\]
Then we can get a better estimate of the norm of 
$\calS_d$:
\[
  \|f\|^2_{\calG_d}\le\sum_\setu d^{|\setu|}\cdot 
    \|f_\setu\|^2_{H_\setu}\cdot C_0^{2\cdot|\setu|}\cdot
     d^{-|\setu|}\le \|f\|^2_{\calH_d}\cdot 
      \max_{k\le d}C_0^{2\cdot k}\cdot d^{-k}.
\]
Since the estimation above is sharp, we conclude that 
\[
  \|\calS_d\|=\max_{k\le d} C_0^k\cdot d^{-k/2}
\]

The class of such approximation problems contains the following 
{\em weighted}-$L_2$ approximation.

\subsubsection{Weighted $L_2$ Approximation}
Let $\rho$ be a given probability density function (p.d.f. for 
short) on $D$. Without loss of generality, suppose that $\rho$ is positive
(a.e.) on $D$. Then the $L_2(\rho_d,D^d)$ space with 
finite
\[
   \|f\|_{L_2(\rho_d,D^d)}^2=\int_{D^d}|f(\bx)|^2\cdot\rho_d(\bx)\,
   \rd\bx,
\]
is a well defined Hilbert space. Here by $\rho_d$ we mean
\[
  \rho_d(\bx)=\prod_{j=1}^d\rho(x_j).
\]
We then take
\[
  \calG_d=L_2(\rho_d,D^d).
\]
It is well known that the continuity of $\calS_1$ is equivalent 
to the continuity of the following integral operator 
\[
  \calW_1:=\calS_1^*\circ\calS_1:H\to H,\qquad 
  \calW_1(f)(x)=\int_d f(y)\cdot K(x,y)\cdot
    \rho(y)\,\rd y,
\]
since then $\|\calS_1\|^2$ is equal to the largest eigenvalue 
of $\calW_1$, i.e.,
\[
  C_0^2=\max\left\{\lambda\ :\ 
    \lambda\in{\rm spect}(\calW_1)\right\}.
\]
Then 
\[
  1\le \|\calS_d\|^2\le \left(1+\frac{C_0^2}{d}\right)^d.
\]

The condition \eqref{new-ass2} is now equivalent to 
\[
    \int_D f(x)\cdot\rho(x)\,\rd x=0\quad\mbox{for all\ }f\in H,
\]
which is satisfied by various spaces of periodic functions.

\subsection{Algorithms, Errors and Cost}
Since problems considered in this paper are defined over 
Hilbert spaces, we can restrict the attention to 
linear algorithms only, see e.g., \cite{TWW88}, of the form 
\[
  \calA_n(f)=\sum_{j=1}^n L_j(f)\cdot a_j, 
\]
where $L_j$ are continuous linear functionals and $a_j\in\calG_d$.
In the worst case setting considered in this paper, the 
error of an algorithm $\calA_n$ is defined by 
\[
   {\rm error}(\calA_n;\calH_d,\calG_d):=\sup_{f\in\calH_d}
     \frac{\|f-\calA_n(f)\|_{\calG_d}}{\|f\|_{\calH_d}}.
\]

So far, in the complexity study of problems with 
finitely many variables, it has been assumed that the cost 
of an algorithm is given by the number $n$ of functional evaluations. 
We believe that, similar to problems with infinitely many variables, 
the cost of computing $L(f)$ should depend on the {\em number 
of active variables} of $L$. 
More precisely, for given $L\in\calH_d^*$, 
let $h_L\in\calH_d$ be its generator, i.e., 
\[
  L(f)=\il f, h_L\ir_{\calH_d}\quad\mbox{for all $f\in\calH_d$}.
\]
Then $h_L=\sum_{\setu\subseteq[1..d]}h_\setu$, 
\[
   {\rm Act}(L):=\bigg|\bigcup\bigg\{\setv\ :\ h_\setv\not=0, \ 
    h_L=\sum_{\setu\subseteq[1..d]}h_\setu\bigg\}\bigg|
\]
is the number of active variables in $L$, and  
the cost of evaluating $L(f)$ is equal to 
\[
    \$({\rm Act}(L)),
\]
where $\$:\naturals_+\to\reals_+$ is a given {\em cost function}. 
The only assumptions that we make at this point are 
\[
  \$(0)\ge1\quad\mbox{and}\quad \$(k)\le \$(k+1) \mbox{\ \ for 
     all\ \ $k\in\naturals$}.
\]
This includes 
\[
  \$(k)=(k+1)^q, \quad \$(k)=e^{q\cdot k},\quad\mbox{and}\quad
  \$(k)=e^{e^{q\cdot k}}
\]
for some $q\ge0$. 
Then the {\em (information) cost} of $\calA_n=\sum_{j=1}^n
L_j(f)\cdot a_j$ is given by 
\[
   {\rm cost}(\calA_n):=\sum_{j=1}^n \$({\rm Act}(L_j)).
\]

The tractability results obtained so far for functions with 
finite numbers of variables correspond to $\$\equiv1$. In our 
opinion, it makes sense to assume that the cost function is 
at least linear, i.e., 
\[
  \$(k)\ge c\cdot(k+1), \quad k \in \naturals.
\]

\subsection{Information Complexity and Tractability}
By {\em (information) complexity} we mean the minimal information 
cost among all algorithms with errors not exceeding a given error 
demand. That is, for $\e\in(0,1)$, 
\[
  {\rm comp}(\e;\calH_d,\calG_d):=\inf\left\{
    {\rm cost}(\calA)\ :\ {\rm error}(\calA;\calH_d,\calG_d)\le\e
      \right\}. 
\]

We now recall the definition of three kinds of tractabilities.
For a detailed discussion of tractability concepts and results, 
we refer to excellent monographs \cite{NW08,NW10}. We stress 
however, that those results pertain to the constant cost function, 
$\$\equiv 1$. 

We say that the problem $\calS_d$ (or more precisely the sequence 
of problems $\calS_d$) is {\em polynomially tractable} if there 
exist $c,p,q\ge0$ such that 
\[
  {\rm comp}(\e;\calH_d,\calG_d)\le c\cdot \frac{d^q}{\e^p} 
   \quad\mbox{for all $\e\in(0,1)$ and $d\in\naturals_+$}.
\]
It is {\em strongly polynomially tractable} iff the above inequality 
holds with $q=0$, and {\em weakly tractable} iff
\[
  \limsup_{d+1/\e\to\infty}
 \frac{\ln\left({\rm comp}(\e;\calH_d,\calG_d)\right)}{d+1/\e} 
   =0.
\]
When the problem is strongly polynomially tractable then 
\[
  p^{\rm str}:=\inf\left\{p\ :\ \sup_{\e,d} \e^p\cdot 
   {\rm comp}(\e;\calH_d,\calG_d)<\infty \right\}
\]
is called the {\em exponent of strong tractability}. 

There is also a concept of {\em quasi-polynomial tractability} 
introduced recently, see \cite{GneWoz11}. It is weaker than 
polynomial tractability and stronger than weak tractability. 
More precisely, the problem is quasi-polynomially tractable 
if there exist $c,t\ge0$ such that 
\[
  {\rm comp}(\e;\calH_d,\calG_d)\le c\cdot\exp\left(t\cdot(1+\ln(d))
     \cdot(1+\ln(1/\e))\right)
   \quad\mbox{for all $\e\in(0,1)$ and $d\in\naturals_+$}.
\]
This means that ${\rm comp}(\e;\calH_d,\calG_d)
\le c\cdot(e\cdot d)^{t\cdot(1+\ln(1/\e))}$. 
The significance of the quasi-polynomial tractability is that for 
some applications, $d$ can be very large but $\e$ need not be very 
small, say 
$\e=10^{-2}$. Then the complexity of the problem is bounded by a 
polynomial in $d$. 

As we shall prove in the next Sections, the problems considered in 
this paper are quasi-polynomially tractable even when the cost 
function $\$$ is exponential in $d$. 

\section{Results}

\subsection{Number of Active Variables} 
We are interested in a number $m=m(\e,d)$ such that, 
for any $f\in\calH_d$, the terms $f_\setu$ with $|\setu|>m$ 
can be neglected, i.e., 
\begin{equation}\label{ess-p1}
   \bigg\|\sum_{|\setu|>m(\e,d)}f_\setu\bigg\|_{\calG_d}
    \le \e\cdot\bigg\|\sum_{|\setu|>m(\e,d)}f_\setu\bigg\|_{\calH_d}.
\end{equation}
Hence, 
to approximate $\calS_d(f)$ with error bounded by $\e\sqrt{2}$, 
it is enough to use algorithms with functionals $L_j$ that have 
${\rm Act}(L_j)\le m(\e,d)$. 

We first find $m(\e,d)$ for the general tensor product space $\calG_d$ 
and next for the special case \eqref{new-ass}. To distinguish 
between the two cases, we will write respectively $m_1=m_1(\e,d)$ and 
$m_2=m_2(\e,d)$ instead of $m=m(\e,d)$. 

\subsubsection{General Case} 
For given $\e\in(0,1)$ and $d\in\naturals_+$, define  
\begin{equation}\label{def-m1}
   m_1=m_1(\e,d):=\min\bigg\{m\ :\ 
  \sum_{k=m+1}^d \newton{d}{k}\cdot\left(\frac{C_0^2}d\right)^k  
   \le\e^2\bigg\}. 
\end{equation}
Of course, $m_1(\e,d)$ is well defined and is bounded by $d$. 

\begin{proposition}
For every $d$, $\e\in(0,1)$, and $f\in\calH_d$, \eqref{ess-p1} 
holds with $m=m_1(\e,d)$ given by \eqref{def-m1}. 
Moreover, $m_1(\e,d)$ is bounded from above by $\min(d,M)$, 
where $M=M(\e)$ is the solution of 
\[
  \frac{(M+1)!}{C_0^{2\cdot(M+1)}}=\frac{e^{C_0^2}}{\e^2}.
\]
In particular, there exists a constant $C_1$ such that 
\[
   m_1(\e,d) \le C_1\cdot \frac{\ln(1/\e)}{\ln(\ln(1/\e))}
   \quad\mbox{for all\ }\e<e^{-e}. 
\]
\end{proposition}
\begin{proof}
Of course, \eqref{ess-p1} holds if $m_1(\e,d)=d$. 
Therefore we consider only the case when $m_1=m_1(\e,d)<d$. We have 
\begin{eqnarray*}
  && \bigg\|\sum_{|\setu|>m_1} f_\setu\bigg\|_{\calG_d}\, 
   \le\, \sum_{|\setu|>m_1} \|f_\setu\|_{\calG_d}
     \,\le\,\sum_{|\setu|>m_1}
     \|f_\setu\|_{H_\setu}\cdot C_0^{|\setu|}\\
   &&\le \bigg[\sum_{|\setu|>m_1}\gamma_{d,\setu}^{-1}\cdot
     \|f_\setu\|_{H_\setu}^2\bigg]^{1/2}\cdot
  \bigg[\sum_{|\setu|>m_1}\gamma_{d,\setu}
    \cdot C_0^{2\cdot|\setu|}\bigg]^{1/2}\\
   &&= \bigg\|\sum_{|\setu|>m_1} f_\setu\bigg\|_{\calH_d}
     \cdot \bigg[\sum_{|\setu|>m_1}\gamma_{d,\setu}
    \cdot C_0^{2\cdot|\setu|}\bigg]^{1/2}\\
  &&  = \bigg\|\sum_{|\setu|>m_1} f_\setu\bigg\|_{\calH_d}
     \cdot\bigg[\sum_{k=m_1+1}^d\newton{d}{k}\cdot 
      d^{-k}\cdot C_0^{2\cdot k}\bigg]^{1/2}\\
  &&\le \bigg\|\sum_{|\setu|>m_1}
    f_\setu\bigg\|_{\calH_d}\cdot\e.
\end{eqnarray*}
This completes the proof of the first part. 
We now estimate the number $m_1(\e,d)$. Observe that, 
for any $m<d$, we have 
\begin{eqnarray*}
 && \sum_{k=m+1}^d\newton{d}{k}\cdot\left(\frac{C_0^2}d\right)^k
 \,=\, \sum_{k=m+1}^d C_0^{2\cdot k}\cdot \frac{d\cdots(d-k+1)}{d^{\,k}\cdot k!}\\
   &&\le\, \sum_{k=m+1}^d \frac{C_0^{2\cdot k}}{k!} 
   \,\le \,\frac{C_0^{2\cdot(m+1)}}{(m+1)!}\sum_{j=0}^\infty C_0^{2\cdot j}
    \frac{(m+1)!}{(m+1+j)!}\\
   &&= \frac{C_0^{2\cdot(m+1)}}{(m+1)!}\sum_{j=0}^\infty 
    \frac{C_0^{2\cdot j}}{j!} / \newton{m+1+j}{j} 
    \le \frac{C_0^{2\cdot(m+1)}\cdot e^{C_0^2}}{(m+1)!}. 
\end{eqnarray*}
This completes the proof. 
\end{proof}

\begin{remark}
One can slightly improve the estimate of $m_1(\e,d)$ by letting 
$M=M(\e)$ to be the minimal integer such that $C_0^2/(M+1)<1$ and 
\[
  (M+1)!/C_0^{2\cdot(M+1)}\ge \frac1{\e^2\cdot(1-C_0^2/(M+1))}.
\]
This is because the last sum in the proof above can be bounded as
follows:
\[
\sum_{j=0}^\infty 
    \frac{C_0^{2\cdot j}}{j!} / \newton{m+1+j}{j}\le
  \sum_{j=0}\left(\frac{C_0^2}{m+1}\right)^j=\frac1{1-C_0^2\cdot(m+1)}.
\]
\end{remark}

We calculated the values of $\lceil M(\e)\rceil$ for $\e=10^{-q}$ with
$q=1,\dots,10$ for the function approximation problem with the 
Wiener kernel on $[0,1]$ and $\rho(x)\equiv1$. Recall that then 
$C_0^2=1/2$. These values are listed in the following table.
\[
  \begin{array}{c||r|r|r|r|r|r|r|r|r|r|r|r|r}
      q & 1 & 2 & 3 & 4 & 5 & 6 & 7 & 8 & 9 & 10 \\
     \hline     \lceil M(10^{-q})\rceil 
        & 3 & 5 & 7 & 8 &10 &11 &13 &14 &15 & 17 
  \end{array}
\]

\subsubsection{Special Case \eqref{new-ass}}
We now investigate the number of active variables under the 
assumption \eqref{new-ass}. Then, for any $k<d$, 
\begin{eqnarray*}
  \bigg\|\sum_{|\setu|> k}f_\setu\bigg\|_{\calG_d}^2
   &\le& C\cdot  \sum_{|\setu|> k}\|f_\setu\|_{\calG_d}^2\\
  &\le&C\cdot\sum_{|\setu|> k} C_0^{2\cdot|\setu|}\cdot
    \gamma_{d,\setu}^{-1}\cdot\gamma_{d,\setu}\cdot
      \|f_\setu\|_{H_\setu}^2\\
  &\le& C\cdot \max_{\ell>k}\left(C_0^{2\cdot\ell}\cdot d^{-\ell}\right)
   \cdot\bigg\|\sum_{|\setu|>k}f_\setu\bigg\|_{\calG_d}^2.  
\end{eqnarray*}
Therefore, for $m_2=m_2(\e,d)$ given by 
\begin{equation}\label{def-m2}
    m_2:=\left\{\begin{array}{ll} 
   0 & \mbox{if $d<C_0^2$ and $(C_0^2/d)^d\le\e^2/C$,}\\
   d & \mbox{if $d<C_0^2$ and $(C_0^2/d)^d>\e^2/C$,}\\
   \min\left(k\ :\ (C_0^2/d)^{k+1}\le \e^2/C\right) & 
    \mbox{otherwise},\end{array}\right.
\end{equation}
we have the following proposition. 

\begin{proposition}
Suppose that \eqref{new-ass} is satisfied. 
For every $d$, $\e\in(0,1)$, and $f\in\calH_d$, 
\eqref{ess-p1} holds with $m=m_2(\e,d)$ 
given by \eqref{def-m2}. Moreover, 
for $d\ge C_0^2$, 
\[
   m_2(\e,d)\le \min\left(d\,,\,
   \left\lceil\frac{\ln(C/\e^2)}{\ln(d/C_0^2)}
     \right\rceil-1\right)\quad\mbox{and}\quad 
    m_2(\e,d)=O\left(\ln^{-1}(d)\right)\quad\mbox{as\ }
    d\to\infty.
\]
\end{proposition}

\subsection{Changing Dimension Algorithm}
We consider in this section 
very special algorithms that are from the family of {\em changing 
dimension algorithms} introduced in \cite{KSWW09} for integration 
and in \cite{WW10a,WW10b} for approximation of functions with 
infinitely many variables. As shown recently in \cite{Was11}, 
these algorithms yield polynomial tractability for 
weighted $L_2$ approximation problems with infinitely many variables 
and general weights that have the decay greater than one. 

These results are not applicable in this paper since the weights
 $\gamma_{d,\setu}=d^{-|\setu|}$ have decay exactly one. 
However, these weights still allow for quasi-polynomial tractability 
and strong polynomial tractability if \eqref{new-ass} holds. 

More precisely, let $\{(\lambda_{1,n},\zeta_{1,n})\}_{n=1}^\infty$ be the 
set of eigenpairs of the operator 
\[
  W_1:=S_1^*\circ S_1:H\to H
\]
for the class $H$ of univariate functions. 
We assume that $\lambda_{1,n}$ are monotonically decreasing to zero 
with a polynomial speed, i.e., that 
\begin{equation}\label{alpha}
  \alpha:={\rm decay}\left(\{\lambda_{1,n}\}_{n=1}^\infty\right)>0.
\end{equation}
Recall that the decay of a sequence of positive numbers $a_n$ is 
defined by 
\[
   {\rm decay}\left(\{a_n\}_{n=1}^\infty\right):=
  \sup\left\{t\ :\ \sum_{n=1}^\infty a_n^{1/t}<\infty\right\}.
\]
For instance, the decay of $a_n=\Theta\left(n^{-\beta}
\cdot\ln^\delta(n)\right)$ 
is equal to $\beta$. We also assume that $\zeta_{1,n}$'s 
form a complete orthonormal system in $H$.
It is well known that the constant $C_0$ is equal to 
the square-root of the largest eigenvalue of $W_1$, i.e., 
\[
  C_0=\sqrt{\lambda_{1,1}}.
\]

\subsubsection{General Case}
Consider the operator 
\[
  W_\setu=S_\setu^*\circ S_\setu:H_\setu\to H_\setu
\]
for the space $H_\setu$. Due to the  tensor product structure 
of $S_\setu$ and $H_\setu$, the eigenpairs of $W_\setu$ 
are provided by the products of 
the eigenpairs for the univariate case. Let 
$\{\lambda_{\setu,n}\}_{n=1}^\infty$ be the set of all the eigenvalues 
of $W_\setu$ listed in the decreasing order, $\lambda_{\setu,n}\ge
\lambda_{\setu,n+1}$. We now use a standard technique to estimate 
these eigenvalues. For that purpose note that, for any 
\[
  \tau>1/\alpha,
\]
we have 
\[
  \sum_{n=1}^\infty\lambda_{\setu,n}^\tau=[L(\tau)]^{|\setu|}
    \quad\mbox{with}\quad 
   L(\tau):=\sum_{n=1}^\infty \lambda_{1,n}^\tau<\infty.
\]
Therefore the $n$th largest eigenvalue $\lambda_{\setu,n}$ satisfies 
\[
  n\cdot \lambda_{\setu,n}^\tau\le [L(\tau)]^{|\setu|}, 
    \quad\mbox{i.e.,}\quad
    \lambda_{\setu,n}\le \frac{[L(\tau)]^{|\setu|/\tau}}{n^{1/\tau}}.
\]
Let $\zeta_{\setu,n}$ be the normalized eigenfunction corresponding 
to the eigenvalues $\lambda_{\setu,n}$. It is well known, see, e.g., 
\cite{TWW88}, that the algorithm 
\[
  A^*_{\setu,n}(f):=\sum_{j=1}^n\il f, \zeta_{\setu,j}\ir_{H_\setu}
   \cdot \zeta_{\setu,j}
\]
have the minimal errors among all algorithms using $n$ functional 
evaluations and
\[
  {\rm error}(A^*_{\setu,n};H_\setu,\calG_\setu)=\sqrt{\lambda_{\setu,n+1}}
   \le \left[\frac{[L(\tau)]^{|\setu|}}{n+1}\right]^{1/(2\cdot\tau)}.
\]
 
Since $H_\setu$ are orthogonal subspaces of $\calH_d$, the 
algorithms $A^*_{\setu,n}$ are naturally extendable to $\calH_d$ and 
\[
  A^*_{\setu,n}\bigg(\sum_{\setv\subseteq[1..d]}f_\setv\bigg)=
   A^*_{\setu,n}(f_\setu).
\]
Moreover, 
\[
  {\rm cost}(A^*_{\setu,n})\le n\cdot\$(|\setu|).
\]

We are ready to define the algorithms $\calA_{\e,d}$ for the weighted 
space $\calH_d$. For $\e\in(0,1)$, let
\begin{equation}\label{def-alg1}
  \calA_{\e,d}(f):=\il f,1\ir_{H_\emptyset}+
    \sum_{1\le|\setu|\le m_1(\e,d)}A^*_{\setu,n_\setu}(f),
\end{equation}
where
\begin{equation}\label{def-alg2}
  n_\setu=n_{\setu,\e}:=\left\lfloor \frac{[L(\tau)]^{|\setu|}}
   {\e_\setu^{2\cdot\tau}}\right\rfloor \quad\mbox{and}\quad
  \e_\setu=\e_{\setu,d}:=\frac{\e\cdot d^{|\setu|/(2(1+\tau))}}
    {\sqrt{R}}
\end{equation}
with
\[
  R=R(\e,d):=\sum_{k=1}^{m_1(\e,d)}\newton{d}{k}\cdot 
    d^{-k\cdot\tau/(1+\tau)}.
\]

Since $n_\setu$ depends on $\setu$ only via $|\setu|$, we will sometimes 
write $n_{|\setu|}$ or $n_\ell$ if $|\setu|=\ell$ instead of 
$n_\setu$. 

Note that 
\[
  R\le\sum_{k=1}^{m_1(\e,d)} \frac{d^k\cdot d^{-k\cdot\tau/(1+\tau)}}{k!}
   = \sum_{k=1}^{m_1(\e,d)} \frac{d^{k/(1+\tau)}}{k!}
   \le m_1(\e,d)\cdot \frac{d^{\ell^*/(1+\tau)}}{(\ell^*)!},
\]
where
\[
   \ell^*=\ell^*(\e,d):=\min\left(m_1(\e,d)\,,\,
    \left\lfloor d^{1/(1+\tau)}\right\rfloor\right).
\]
This follows from the fact that the sequence 
$d^{k/(1+\tau)}/{k!}$ increases until $k\le d^{1/(1+\tau)}$, 
and next starts to decrease, as can be easily verified. Hence 
\[
  R^{1+\tau}\le \left\{\begin{array}{ll} 
   d^{m_1(\e,d)}/((m_1(\e,d)-1)!)^{1+\tau} & \mbox{for\ }
  d> (m_1(\e,d))^{1+\tau},\\ \ \\
    m_1(\e,d)\cdot e^{m_1(\e,d)}  & 
    \mbox{otherwise},
   \end{array}\right.
\]
where in the second case we replaced $(\ell^*)!$ by 
$(\ell^*/e)^{\ell^*}$ and used the fact that $\ell^*\le m_1(\e,d)$. 
This means that 
\[
   (R(\e,d))^{1+\tau}\le C_1\frac{\ln(1/\e)}{\ln(\ln(1/\e))}
    \cdot\e^{-C_1/\ln(\ln(1/\e))} \qquad\mbox{if}\quad 
    m_1(\e,d)\ge d^{1/(1+\tau)}
\]
and 
\[     
   (R(\e,d))^{1+\tau} \le \e^{\,C_1/[(1+\tau)\cdot\ln(\ln(1/\e))]}
    \cdot d^{\,C_1\cdot\ln(1/\e)/\ln(\ln(1/\e))} 
    \qquad\mbox{if}\quad m_1(\e,d)<d.
\]
Of course, in all the above estimates, we assume that 
$\e<e^{-e}$. 

We now estimate the error of the algorithm $\calA_{\e,d}$. Since 
$\calA_{\e,d}(f_\setu)=0$ for all $f_\setu$ with
$|\setu|>m_1(\e,d)$, we have 
\[
  \left[{\rm error}(\calA_{\e,d};\calH_d,\calG_d)\right]^2 =
  \sum_{1\le|\setu|\le m_1(\e,d)}\gamma_{d,\setu}\cdot
    \left[{\rm error}(A^*_{\setu,n_\setu};H_\setu,\calG_d)\right]^2 
    +  \sum_{|\setu|>m_1(\e,d)}\gamma_{d,\setu}\cdot  C_0^{2\cdot|\setu|}
\]
The latter sum satisfies 
\[
  \sum_{|\setu|>m_1(\e,d)}\gamma_{d,\setu}\cdot 
     C_0^{2\cdot|\setu|}=\sum_{k=m_1(\e,d)+1}^d\newton{d}{k}\cdot 
    \left(\frac{C_0^2}d\right)^k\le \e^2,
\]
whereas the former sum is bounded by 
\begin{eqnarray*}
  \sum_{1\le|\setu|\le m_1(\e,d)}\gamma_{d,\setu}\cdot 
   \e_\setu^2 & =& \sum_{\ell=1}^{m_1(\e,d)} \newton{d}{l}\cdot d^{-\ell}
    \cdot\e_\ell^2\\
  & =& \e^2\cdot R^{-1}\cdot \sum_{\ell=1}^{m_1(\e,d)}\newton{d}{\ell}
      \cdot d^{-\ell}\cdot d^{\ell/(1+\tau)}\\
   &=&\e^2.
\end{eqnarray*}
This means that 
\[
  {\rm error}(\calA_{\e,d}; \calS_d,\calH_d)\le \e\cdot\sqrt{2}. 
\]

We now estimate the cost of $\calA_{\e,d}$:
\[
  {\rm cost}(\calA_{\e,d})\le \$(0)+\sum_{1\le|\setu|\le m_1(\e,d)}
  \$(|\setu|)\cdot n_\setu\le
   \$(0)+\$(m_1(\e,d))\sum_{1\le|\setu|\le m_1(\e,d)} n_\setu
\]
and 
\begin{eqnarray*}
  \sum_{1\le|\setu|\le m_1(\e,d)} n_\setu &=&
   \sum_{\ell=1}^{m_1(\e,d)}\newton{d}{\ell}\cdot n_\ell\,
  \le\, \e^{-2\cdot\tau}\cdot R^\tau\cdot 
   \sum_{\ell=1}^{m_1(\e,d)}\newton{d}{\ell}\cdot 
   \frac{[L(\tau)]^\ell}{d^{\ell\cdot\tau/(1+\tau)}}\\
   &\le& \max\left(L(\tau), [L(\tau)]^{m_1(\e,d)}\right)\cdot 
      \frac{R^{1+\tau}}{\e^{2\cdot\tau}}.
\end{eqnarray*}

We summarize this in the following theorem.

\begin{theorem}\label{thm:main}
Suppose that \eqref{alpha} holds. The approximation problem 
is quasi-polynomially tractable 
even if $\$$ is an exponential function of $d$, 
$\$(d)=O\left(e^{q\cdot d}\right)$,  and is weakly tractable even if 
$\$(d)=O\left(e^{e^{q\cdot d}}\right)$ for some $q\ge0$. 
Moreover, for any $\tau>1/\alpha$, 
the algorithms $\calA_{\e,d}$ have errors bounded by 
$\e\cdot\sqrt{2}$ and cost bounded by 
\[
  {\rm cost}(\calA_{\e,d})\le\$(0)+\$(m_1(\e,d))\cdot \max\left(
   L(\tau),[L(\tau)]^{m_1(\e,d)}\right)\cdot 
    \frac{[R(\e,d)]^{1+\tau}}{\e^{2\cdot\tau}},
\]
where $m_1(\e,d)$ is given by \eqref{def-m1}, e.g., 
\[
   m_1(\e,d)\le C_1\cdot\frac{\ln(1/\e)}{\ln(\ln(1/\e))} 
   \quad\mbox{for\ }\e<e^{-e}, 
\]
and 
\[
  [R(\e,d)]^{1+\tau}\le\left\{\begin{array}{ll}
  C_1\frac{\ln(1/\e)}{\ln(\ln(1/\e))}
    \cdot\e^{-C_1/\ln(\ln(1/\e))} &\mbox{if}\quad 
    m_1(\e,d)\ge d^{1/(1+\tau)},\\ 
    \ \\
   \e^{\,C_1/[(1+\tau)\cdot\ln(\ln(1/\e))]}
    \cdot d^{\,C_1\cdot\ln(1/\e)/\ln(\ln(1/\e))} & 
    \mbox{otherwise}.
    \end{array}\right.
\]
\end{theorem}

We believe that the result on quasi-polynomial tractability is sharp 
in general, i.e., there exist $H$ and $G_1$ such that the 
corresponding multivariate problem with weights $\gamma_{d,\setu}=
d^{-|\setu|}$ is only quasi-polynomially tractable. 
However, as we prove in the next section, Theorem \ref{thm:main} is not sharp 
when \eqref{new-ass} holds. 

\subsubsection{Special Case \eqref{new-ass}}
We begin this section by assuming for a moment that 
\begin{equation}\label{new-ass1}
  \bigg\|\sum_{\setu\subseteq[1..d]}f_\setu\bigg\|_{\calG_d}^2=
  \sum_{\setu\subseteq[1..d]}\|f_\setu\|_{\calG_d}^2
   \quad\mbox{for every}\quad f\in\calH_d,
\end{equation}
which is a stronger assumption than \eqref{new-ass}. 
Similar spaces with norms satisfying \eqref{new-ass1} 
have been considered in \cite{WW10a,WW10b} 
for functions with infinitely many variables ($d=\infty$) 
and some of the results below follow from \cite{WW10a}. 

As shown in \cite{WW10a}, \eqref{new-ass1} allows for a simple
characterization of the spectrum of 
\[
  \calW_d=\calS_d^*\circ\calS_d:\calH_d\to\calH_d
\]
in terms of the spectrum of $\calW_1$. 
Indeed, the eigenvalues of $\calW_d$ are given by 
\[
  \gamma_\setu\cdot\prod_{j\in\setu}\lambda_{1,k_j}
\]
for all $\setu$ and $k_j\in\naturals$. For $\setu=\emptyset$, 
1 is the corresponding eigenvalue. 

Let $\lambda_{d,n}$ ($n\in\naturals_+$) be the eigenvalues of 
$\calW_d$ ordered so that 
\[
   \lambda_{d,n}\ge\lambda_{d,n+1}\quad\mbox{for all $n$}.
\]
Let $\eta_{d,n}$ be the corresponding eigenfunctions that form 
a complete orthonormal system in $\calH_d$. They also have a tensor 
product form and $\eta_{d,n}$ corresponding to the eigenvalue 
$\gamma_\setu\prod_{j\in\setu}\lambda_{1,k_j}$ has all the active 
variables listed in $\setu$. 

Define
\[
  \calA_{\e,d}^{*}(f):=\sum_{j=1}^{n(\e,d)} \il f,
  \eta_{d,j}\ir_{\calH_d} \cdot \eta_{d,j}
   \quad\mbox{with}\quad n(\e,d):=\min\left\{k\ :\ 
    \lambda_{d,k+1}\le \e^2\right\}.
\]
It follows from \cite{WW10a} that $\calA^*_{\e,d}$ is optimal for any 
cost function $\$$, i.e., 
${\rm error}(\calA^*_{\e,d};\calH_d,\calH_d)\le\e$ and 
\[
   {\rm cost}(\calA^*_{\e,d})=\min\left\{{\rm cost}(\calA)\ :\ 
     {\rm error}(\calA;\calH_d,\calG_d)\le\e\right\}=
    {\rm comp}(\e;\calH_d,\calG_d).
\]
Now $\lambda_{1,1}=C_0^2$, 
\[
   m_2(\e,d)=\min\left(d\,,\,
  \left\lceil\frac{\ln(1/\e^2)}{\ln(d/\lambda_{1,1})}
   \right\rceil-1\right),
\]
and the functional evaluations $\il f,\eta_{d,j}\ir_{\calH_d}$ 
used by the algorithm have at most $m_2(\e,d)$ active variables.

Note that, for every $\delta>0$, we can bound $m_2(\e,d)$ by 
\begin{equation}\label{brbr}
  m_2(\e,d)\le \max\left(\lambda_{1,1}\cdot e^{1/\delta}\,,\,
    \delta\cdot\ln(1/\e^2)\right).
\end{equation}
Indeed, \eqref{brbr} trivially holds if $d\le \lambda_{1,1}\cdot
e^{1/\delta}$, and 
\[
  \frac{\ln(1/\e)}{\ln(d/\lambda_{1,1})}< 
     \frac{\ln(1/\e)}{\ln(e^{1/\delta})}=\delta\cdot \ln(1/\e^2)
   \quad\mbox{if}\quad d>\lambda_{1,1}\cdot e^{1/\delta}. 
\]

Take now 
\[
    \tau>1/\alpha,
\]
where, as before, $\alpha={\rm decay}(\{\lambda_{1,n}\}_{n=1}^\infty)>0$. 
Using a standard technique, we get 
\begin{eqnarray*}
  \lambda_{d,k}^\tau\cdot k&\le& \sum_{j=1}^\infty \lambda_{d,j}^\tau
    \,=\,\sum_{\setu\subseteq[1..d]}\gamma_{d,\setu}^\tau
       \sum_{\bk\in\naturals^{|\setu|}_+}\prod_{\ell=1}^{|\setu|}
     \lambda_{1,k_\ell}^\tau\\
    &=& \sum_{\ell=0}^d\newton{d}{\ell}\cdot \left(\frac{L(\tau)}
       {d^\tau}\right)^\ell\,=\, 
        \left(1+\frac{L(\tau)}{d^\tau}\right)^d\\
    &\le& e^{L(\tau)\cdot d^{1-\tau}}\,<\,\infty.
\end{eqnarray*}
Hence 
\[
  \lambda_{d,k}\le e^{L(\tau)\cdot d^{1-\tau}/\tau}\cdot k^{-1/\tau}
   \quad\mbox{and}\quad 
   n(\e,d)\le \left\lceil e^{L(\tau)\cdot d^{1-\tau}}\cdot
     \e^{-2\cdot\tau}\right\rceil - 1.    
\]
Note that the term $e^{L(\tau)\cdot d^{1-\tau}}$ is 
bounded from above by $e^{L(\tau)}$ if $\tau\ge1$, and converges to $1$ 
with increasing $d$ if $\tau>1$. 

We return now to the original assumption \eqref{new-ass}. By replacing 
$\e$ by $\e/\sqrt{C}$ in all the formulas above, we get that 
$\calA^*_{\e/\sqrt{C},d}$
has the error bounded by $\e$ when the norm in $\calG_d$ satisfies
\[
  \bigg\|\sum_{\setu\subseteq[1..d]}f_\setu\bigg\|_{\calG_d}^2=
   C\cdot\bigg\|\sum_{\setu\subseteq[1..d]}f_\setu\bigg\|_{\calH_d}^2
   \quad \mbox{for all $f\in\calH_d$}.
\]
Moreover, all upper bounds on the cost and errors provide 
corresponding upper bounds for norms that satisfy only
\eqref{new-ass}, i.e., when the above equality is replaced by 
inequality.

This yields the following theorem. 

\begin{theorem}\label{thm:orth}
Suppose that \eqref{new-ass} and \eqref{alpha} hold. 
Then for any $\tau>1/\alpha$, 
\[
  {\rm comp}(\e;\calS_d,\calH_d)\le
    \$\left(m_2(\e/\sqrt{C},d)\right)
      \cdot\frac{e^{L(\tau)\cdot d^{1-\tau}}}{(\e/\sqrt{C})^{2\cdot\tau}} 
\]
with 
\[
   m_2(\e/\sqrt{C},d)\le\min\left(d\,,\,
     \frac{\ln(C/\e^2)}{\ln(d/\lambda_{1,1})}\right).  
\]
Due to \eqref{brbr}, the approximation problem is strongly
polynomially tractable with the exponent 
\[
  p^{\rm str}\le 2\cdot \max(1,1/\alpha).
\]
even if $\$(d)=O\left(e^{q\cdot d}\right)$, and is 
weakly tractable even if $\$(d)=O\left(e^{e^{q\cdot d}}\right)$ 
for some $q\ge0$. 
\end{theorem}

We now show that the upper bound on $p^{\rm str}$ is sharp if 
\eqref{new-ass1} holds. 

\begin{proposition}\label{prop:orth}
If \eqref{new-ass1} and \eqref{alpha} hold then 
\[
   p^{\rm str}= 2\cdot \max(1,1/\alpha).
\]
\end{proposition}
\begin{proof}
Even for $d=1$, we have ${\rm comp}(\e;H,\calG_1)=
\Omega\left(\e^{-2/\alpha}\right)$. Hence we only need to consider 
the case $\alpha>1$. 
Suppose by the contrary that $p^{\rm str}=p$ for 
$p<2$. The complexity of the problem with 
any cost function $\$$ satisfying our assumptions is bounded from below
by the complexity when $\$(d)=1$ for all $d$, and the latter
complexity is fully determined by the eigenvalues of $\calW_d$. That
is, we have 
\[
    {\rm comp}(\e;\calH_d,\calG_d)\ge\$(0)\cdot 
    \min\left\{k\ :\ \lambda_{d,k+1}\le\e^2\right\}
\]
for any cost function $\$$. 
Take any $\widehat{p}\in(p,2)$. Then there is $c(\widehat{p})$ 
such that 
\[
  \lambda_{d,k}\le c(\widehat{p})\cdot k^{-2/\widehat{p}}
   \quad\mbox{for all $\e<1$ and $d\ge1$}. 
\]
However, then, for any $q>\widehat{p}$, 
\begin{equation}\label{cos}
   \sum_{k=1}^\infty\lambda_{d,k}^{q/2}\le (c(\widehat{p}))^{q/2}\cdot
   \sum_{k=1}^\infty k^{-q/\widehat{p}}<\infty \quad\mbox{for all $d\ge1$}.
\end{equation}
Take $q\in(\widehat{p},2)$. As already explained,
\[
  \sum_{k=1}^\infty\lambda_{d,k}^{q/2}=
    \left(1+\frac{L(q/2)}{d^{q/2}}\right)^d,
\]
which converges to infinity as $d\to\infty$. This contradicts 
\eqref{cos} and completes the proof. 
\end{proof}

We apply Theorem \ref{thm:orth} to the following $L_2$ approximation 
problem. 

\medskip\noindent{\bf Example.\ } Consider $K(x,y)=\min(x,y)$, 
$D=[0,1]$, and $\rho\equiv1$. It is well known that for the 
corresponding $L_2$ approximation problem, we have 
$\alpha=2$. Hence, for $\$(d)$ at most exponential 
in $d$, we have strong tractability with the exponent
\[
   p^{\rm str}\le 2. 
\]

\section*{Acknowledgments}
The research presented in the paper was initiated during our stay
at the {\em Hausdorff Research Institute for Mathematics},
University of Bonn, Summer 2011. We would like to thank the Institute and
the organizers of
{\em Analysis and Numerics of High Dimensional Problems} Trimester
Program for their hospitality.

M.~Hegland, 
Mathematical Sciences Institute, 
The Australian National University, 
Canberra ACT 0200 Australia, email: markus.hegland@anu.edu.au

\medskip\noindent
G.~W.~Wasilkowski, Department of Computer Science, University of
Kentucky, Lexington, KY 40506, USA, email: greg@cs.uky.edu


\begin{thebibliography}{99}

\setlength{\parsep }{-0.5ex}
\setlength{\itemsep}{-0.5ex}
\newcommand\BAMS{\emph{Bull. Amer. Math. Soc.\ }}
\newcommand\BIT{\emph{BIT\ }}
\newcommand\Com{\emph{Computing\ }}
\newcommand\CA{\emph{Constr. Approx.\ }}
\newcommand\FCM{\emph{Found. Comput. Math.\ }}
\newcommand\JAT{\emph{J. Approx. Th.\ }}
\newcommand\JC{\emph{J. Complexity\ }}
\newcommand\JCP{\emph{J. of Computational Physics\ }}
\newcommand\JMA{\emph{SIAM J. Math. Anal.\ }}
\newcommand\JMAA{\emph{J. Math. Anal. Appl.\ }}
\newcommand\JMM{\emph{J. Math. Mech.\ }}
\newcommand\JMP{\emph{J. Math. Physics\ }}
\newcommand\MC{\emph{Math. Comp.\ }}
\newcommand\NM{\emph{Numer. Math.\ }}
\newcommand\RMJ{\emph{Rocky Mt. J. Math.\ }}
\newcommand\SJNA{\emph{SIAM J. Numer. Anal.\ }}
\newcommand\SR{\emph{SIAM Rev.\ }}
\newcommand\TAMS{\emph{Trans. Amer. Math. Soc.\ }}
\newcommand\TCS{\emph{Theoretical Computer Science\ }}
\newcommand\TOMS{\emph{ACM Trans. Math. Software\ }}
\newcommand\USSR{\emph{USSR Comput. Maths. Math. Phys.\ }}
\frenchspacing
\bibitem{CDMK}
J.~Creutzig, S.~Dereich, T.~M\"uller-Gronbach, and K.~Ritter,
Infinite-dimensional quadrature and approximation of distributions,
EMS Tracs in Mathematics Vol. {\bf 12}, Z\"urich, 2010.

\bibitem{Gne10}
M. Gnewuch, Infinite-dimensional integration on weighted 
Hilbert spaces, submitted. 

\bibitem{GneWoz11}
M. Gnewuch and H. Wo\'zniakowski, 
Quasi-polynomial tractability, 
\JC {\bf 27} (2011), 312-330.

\bibitem{HegPes05}
M. Hegland and V. Pestov,
Additive models in high dimension, 
{\em ANZIAM J.} {\bf 46} (2005), 1205-1221.

\bibitem{HMNR10}
F. J. Hickernell, T. M\"uller-Gronbach, B. Niu, K. Ritter,
Multi-level Monte Carlo algorithms for infinite-dimensional integration 
on $\reals^\natural$, \JC {\bf 26} (2010), 229-254. 

\bibitem{KSWW09}
F. Y. Kuo, I. H. Sloan, G. W. Wasilkowski and H. Wo\'zniakowski, 
Liberating the dimension, \JC {\bf 26} (2010), 422-454.
     DOI: 10.1016/j.jco.2009.12.003.

\bibitem{KSWW09b}
F. Y. Kuo, I. H. Sloan, G. W. Wasilkowski and H. Wo\'zniakowski, 
On decompositions of multivariate functions,
{\em Mathematics of Computation} {\bf 79} (2010), 953-966.
DOI: 0.1090/S0025-5718-09-02319-9

\bibitem{Levy}
P. L\'evy, {\em Problemes Concrets D'Analyse, Functionnelle,}
Gauthier-Villars, Paris 1951. 

\bibitem{NH10}
B.~Niu and F.~J.~Hickernell, 
Monte Carlo simulation of stochastic integrals when the cost 
function evaluation is dimension dependent, 
{\em Monte Carlo and Quasi-Monte Carlo Methods 2008} (P.L'Ecuyer 
and A. B. Owen, eds.,), pp.\,545-560, Springer Verlag, Berlin 2010. 

\bibitem{NHMR10}
B. Niu and F. J. Hickernell, T. M\"uller-Gronbach, and K. Ritter,
Deterministic multi-level algorithms for infinite-dimensional 
integration on $\reals^\natural$, \JC {\bf 27} (2011), 331-351. 
DOI: 10.10.16/jco.2010.08.001

\bibitem{NW08}
E. Novak and H. Wo\'zniakowski, 
\emph{Tractability of Multivariate Problems, Volume I: Linear Information,}
EMS Tracs in Mathematics Vol. {\bf 6}, Z\"urich, 2008. 

\bibitem{NW10}
E. Novak and H. Wo\'zniakowski,
\emph{Tractability of Multivariate Problems, Volume II: 
Standard Information for Functionals,}
EMS Tracs in Mathematics Vol. {\bf 12}, Z\"urich, 2010.

\bibitem{PlaWas10}
L. Plaskota and G. W. Wasilkowski, 
Tractability of infinite-dimensional integration in the worst
case and randomized settings, \JC {\bf 27} (2011), 505-518.

\bibitem{TWW88}
J. F. Traub, G. W. Wasilkowski, and H. Wo\'zniakowski,
\emph{Information-Based Complexity}, Academic Press, New York, 1988.

\bibitem{Was11}
G. W. Wasilkowski, 
Liberating the dimension for $L_2$ approximation, \JC to appear. 

\bibitem{WW10a}
G. W. Wasilkowski and H. Wo\'zniakowski,
Liberating the dimension for function approximation, 
\JC {\bf 27} (2011), 86-110. 
DOI: 10.10.16/jco.2010.08.004

\bibitem{WW10b}
G. W. Wasilkowski and H. Wo\'zniakowski,
Liberating the dimension for function approximation: 
standard information, \JC {\bf 27} (2011), 417-440.
DOI: 10.10.16/j.jco.2011.02.002

\end{thebibliography}
\end{document}